\documentclass[11pt]{amsart}

\usepackage{amsmath}
\usepackage[pagebackref=true, colorlinks=true, citecolor=blue, pdfborder={0 0 .1},breaklinks]{hyperref}

\author[D.M. Ambrose]{David M. Ambrose}
\address{Drexel University, Department of Mathematics, Philadelphia, PA 19104, USA}
\email{dma68@drexel.edu}
\author[M.C. Lopes Filho]{Milton C. Lopes Filho}
\address{Instituto de Matematica, Universidade Federal do Rio de Janeiro, Caixa Postal 68530,
Rio de Janeiro, RJ, 21941-909 Brazil}
\email{mlopes@im.ufrj.br}
\author[H.J. Nussenzveig Lopes]{Helena J. Nussenzveig Lopes}
\address{Instituto de Matematica, Universidade Federal do Rio de Janeiro, Caixa Postal 68530,
Rio de Janeiro, RJ, 21941-909 Brazil}
\email{hlopes@im.ufrj.br}

\title[Kuramoto-Sivashinsky with singular data]
{Existence and analyticity of solutions of the Kuramoto-Sivashinsky equation with singular data}

\newtheorem{theorem}{Theorem}
\newtheorem{lemma}[theorem]{Lemma}
\newtheorem{remark}[theorem]{Remark}

\newcommand{\vertiii}[1]{{\left\vert\kern-0.25ex\left\vert\kern-0.25ex\left\vert #1
    \right\vert\kern-0.25ex\right\vert\kern-0.25ex\right\vert}}

\newcommand{\calX}{\mathcal{X}}
\newcommand{\calY}{\mathcal{Y}}
\newcommand{\calPM}{\mathcal{PM}}
\newcommand{\calL}{\mathcal{L}}
\newcommand{\calB}{\mathcal{B}}
\newcommand{\real}{\mathbb{R}}

\begin{document}

\begin{abstract} We prove existence of solutions to the Kuramoto-Sivashinsky equation with low-regularity data,
in function spaces based on the Wiener algebra and in pseudomeasure spaces.  In any spatial dimension,
we allow the data to have its antiderivative in the Wiener algebra.  In one spatial dimension, we also allow
data which is in a pseudomeasure space of negative order.  In two spatial dimensions, we also allow data which
is in a pseudomeasure space one derivative more regular than in the one-dimensional case.  In the course of
carrying out the existence arguments, we show a parabolic gain of regularity of the solutions as compared to the data.
Subsequently, we show that the solutions are in fact analytic at any positive time in the interval of existence.
\end{abstract}

\maketitle

\section{Introduction}

The scalar form of the Kuramoto-Sivashinsky equation is
\begin{equation}\label{KSEquationOriginal}
\phi_{t}+\frac{1}{2}|\nabla\phi|^{2}+\Delta^{2}\phi+\Delta\phi=0.
\end{equation}
This is taken with initial data
\begin{equation}\label{data}
\phi(\cdot,0)=\phi_{0}.
\end{equation}
The spatial domain we consider is the $n$-dimensional torus $\mathbb{T}^{n},$ which is given by
\begin{equation}\nonumber
\mathbb{T}^{n}=\prod_{i=1}^{n}\left[0,L_{i}\right],
\end{equation}
for some given lengths $L_{i}>0,$ $i\in\{1,\ldots,n\},$ and with periodic boundary conditions.
Equation \eqref{KSEquationOriginal} was introduced separately by Kuramoto and Tsuzuki when studying
pattern formation in reaction-diffusion equations \cite{kuramoto} and by Sivashinsky in modeling the
evolution of flame fronts \cite{sivashinsky}.  As a model of flame fronts, the surface $(\vec{x},\phi(\vec{x},t))$
gives the location of the interface between the burnt and unburnt phases of a gas undergoing combustion.
As such, the physical cases are $n=1$ (representing the interface between two two-dimensional gases) and
$n=2$ (representing the interface between two three-dimensional gases).

Demonstrating local well-posedness of the initial value problem \eqref{KSEquationOriginal}, \eqref{data} is
straightforward for relatively smooth data.  For global well-posedness, the situation is only clear in dimension
$n=1.$  In this case the nonlinearity has a simpler structure, and there are many results, especially that of Tadmor
\cite{tadmor}, but also those of Bronski and Gambill \cite{bronskiGambill}, Goodman \cite{goodman}, and
Nicolaenko, Scheurer, and Temam \cite{NST}.  These papers all assume the same regularity on the initial data, which
is that $\phi_{0}\in H^{1}.$

In two space dimensions, there are two types of global existence results, both of which have limitations.  The
earliest global existence result in two dimensions was the thin domain result of Sell and Taboado \cite{sellTaboada};
this was then followed by the other thin-domain results \cite{benachourEtAl},   \cite{kukavicaMassatt}, \cite{molinetThin}.
Other than these, the first author and Mazzucato have demonstrated global existence of small solutions for the
two-dimensional Kuramoto-Sivashinsky equation for certain domain sizes (i.e. placing certain conditions on $L_{1}$ and
$L_{2}$), but without the anisotropy inherent in the thin-domain results \cite{ambroseMazzucato1},
\cite{ambroseMazzucato2}.

Other global results rely upon modifying either the linear or nonlinear parts of \eqref{KSEquationOriginal}.
For instance, by no longer considering fourth-order
linear terms a maximum principle may be introduced, leading to global existence
of solutions \cite{lariosYamazaki}, \cite{molinetBS}.   Changing the power in the nonlinear term leads to global existence
or singularity formation, depending on the power, as demonstrated in \cite{belloutEtAl}.  Global existence also follows
from the introduction of appropriate transport terms, as shown in \cite{cotiZelatiEtAl}, \cite{fengMazzucato}.

Gruji\'{c} and Kukavica demonstrated existence of solutions for the Kuramoto-Sivashinsky equation in one dimension,
with $\partial_x \phi_0 \in L^{\infty},$ and also demonstrated analyticity of the solutions at positive times \cite{grujicKukavica}.
Biswas and Swanson considered the Kuramoto-Sivashinsky equation in general dimension. Their results include
improving the assumption made by Gruji\'{c} and Kukavica, in dimension one,  on the regularity of the data (by one);  Biswas and Swanson also study higher regularity through estimates of Gevrey norms \cite{biswasSwanson}.

This research is naturally related to work on the Navier-Stokes equations, for which there have been many studies of existence
of solutions starting from low-regularity data. The optimal result in critical spaces is due to Koch and Tataru for
data in $BMO^{-1}$ \cite{kochTataru}.  The
present work draws more from other studies, such as by Cannone and Karch for data in $PM^{2},$ and by Lei
and Lin for data in $X^{-1}$ \cite{cannoneKarch,leiLin} (see Sections \ref{preliminarySection} and \ref{pseudomeasureSection} in the present work for the definition of these spaces).  In \cite{baePAMS}, Bae proved a version of the Lei-Lin result using a
two-norm approach, which also gives analyticity of the solution at positive times, and which drew
upon the earlier work \cite{baeBiswasTadmor}. The authors of the present paper adapted the work of \cite{baePAMS} to
the spatially periodic case, finding an improved estimate for the radius of analyticity \cite{ALN4}.

The primary contribution of the present work is to weaken the assumed regularity of the initial data
as compared to prior works on existence of solutions for the Kuramoto-Sivashinsky equation
\eqref{KSEquationOriginal}. Our primary motivation is to examine how the two-norm approach may be used to improve
regularity requirements and analyticity estimates beyond the Navier-Stokes system.

The first author and
Mazzucato proved existence of solutions for the two-dimensional Kuramoto-Sivashinsky equations in the case
of small domain sizes with data which has one derivative in the Wiener algebra
or one derivative in $L^{2}$ \cite{ambroseMazzucato1},
\cite{ambroseMazzucato2}.  Subsequently, Coti Zelati, Dolce, Feng, and Mazzucato treated situations (for an
equation with added advection) with data in $L^{2}$ \cite{cotiZelatiEtAl}; Feng and Mazzucato also treated
a different class of advective equations, again with $L^{2}$ data, in \cite{fengMazzucato}.
Biswas and Swanson treat the whole-space case rather than the spatially periodic case, and take data such that the
Fourier transform is in an $L^{p}$ space \cite{biswasSwanson}, with $p\neq 1$ and $p\neq\infty$ (we treat the
complementary cases of periodic data with Fourier coefficients in $\ell^{1}$ or $\ell^{\infty}$).
Existence of solutions for the Kuramoto-Sivashinsky equation with pseudomeasure data was treated by Miao and
Yuan, but only in non-physical spatial dimensions, specificially $n=4,$ $n=5,$ and $n=6$ \cite{miaoYuan}.
In the present work, we deal with the physically relevant spatial dimensions $n=1$ and $n=2.$
It is notable that our one-dimensional
existence theorem allows initial data with Fourier coefficients which grow as the Fourier variable, $k,$ goes to infinity,

We prove that our solutions are global in time in the case that the linearized problem has no growing Fourier modes.
This amounts to an assumption of smallness of the periodic cell that comprises the spatial domain.  In the general case
of larger period cells, our results are valid up until a finite time.  This is consistent with the lack of general global existence
theory for the Kuramoto-Sivashinsky equation in dimension two and higher.
In addition to proving that solutions exist, we also prove that they are analytic at positive times, following the approach of
Bae \cite{baePAMS}, which the authors also used previously for the Navier-Stokes equations \cite{ALN4}.

The plan of the paper is as follows.  We establish some preliminaries in Section \ref{preliminarySection}.  This includes
introducing a number of function spaces, and giving an abstract fixed point result.  In Section
\ref{wienerDataBilinearSection} we establish existence of solutions with data in a space related to the Wiener algebra.
In Section \ref{pseudomeasureSection},
we treat existence of solutions with data in pseudomeasure spaces.  We establish the associated
linear estimates in Section \ref{pseudomeasureLinear},
the nonlinear estimates in one spatial dimension in Section \ref{1DSection},
and the nonlinear estimates in two spatial dimensions in Section \ref{2DSection}.
Analyticity of all of these solutions at positive times is demonstrated in Section \ref{analyticitySection}.
The main theorems are the existence theorems Theorem \ref{existY-1} at the beginning of Section
\ref{wienerDataBilinearSection}, Theorem \ref{n=1PseudomeasureTheorem} at the beginning of Section
\ref{1DSection}, and Theorem \ref{n=2PseudomeasureTheorem} at the beginning of Section \ref{2DSection},
and the analyticity theorems Theorem \ref{firstAnalyticityTheorem} and
Theorem \ref{secondAnalyticityTheorem} at the end of Section \ref{analyticitySection}.
We close with some concluding remarks in Section \ref{conclusionSection}.

\section{Preliminaries}\label{preliminarySection}

We observe that the mean of $\phi$ does not influence the evolution of $\phi$. We thus introduce the projection $\mathbb{P}$ which
removes the mean of a periodic function, as follows:
\begin{equation}\nonumber
\mathbb{P}f=f-\frac{1}{L_{1}\cdots L_{n}}\int_{\mathbb{T}^{n}}f(x)\ dx.
\end{equation}
We let $\psi=\mathbb{P}\phi,$ and we note that $\nabla\psi=\nabla\phi;$ we then see that $\psi$ satisfies the equation
\begin{equation}\label{KSEquationPsi}
\psi_{t}+\frac{1}{2}\mathbb{P}|\nabla\psi|^{2}+\Delta^{2}\psi+\Delta\psi=0.
\end{equation}

Recall the Fourier series of a periodic function, given in terms of its Fourier coefficients:
\begin{equation}\nonumber
f(x) \sim \sum_{k\in\mathbb{Z}^{n}} \hat{f}(k)e^{2\pi i k_{1}x_{1}/L_{1}}\cdots e^{2\pi i k_{n}x_{n}/L_{n}}.
\end{equation}
From this we see directly that the symbol of the partial differential operator $\partial_{x_{i}}$ is
\begin{equation}\nonumber
\sigma(\partial_{x_{j}})=\frac{2\pi i}{L_{j}}k_{j}.
\end{equation}
Therefore the symbol of the Laplacian and bi-Laplacian are
\begin{equation}\nonumber
\sigma(\Delta)=-\sum_{j=1}^{n}\left(\frac{2\pi}{L_{j}}\right)^{2}k_{j}^{2},
\end{equation}
\begin{equation}\nonumber
\sigma(\Delta^{2})=\left(\sum_{j=1}^{n}\left(\frac{2\pi}{L_{j}}\right)^{2}k_{j}^{2}\right)^{2}.
\end{equation}

We next introduce some spaces based on the Wiener algebra, which we denote as $Y^{m}$ for $m\in\mathbb{R}.$
A periodic function, $f,$ is in $Y^{m}$ if the norm given by
\begin{equation}\nonumber
\|f\|_{Y^{m}}=|\hat{f}(0)|+\sum_{k\in\mathbb{Z}^{n}_{*}}|k|^{m}|\hat{f}(k)|
\end{equation}
is finite.  If $m=0,$ then this space is exactly the Wiener algebra.
We let $T>0$ be given, with $T$ possibly being infinite.  On the space-time domain $[0,T]\times\mathbb{T}^{n},$
we also have a related function space, $\mathcal{Y}^{m}.$  The norm for this space is
\begin{equation}\nonumber
\|f\|_{\mathcal{Y}^{m}}=\sup_{t\in[0,T]}|\hat{f}(t,0)|+\sum_{k\in\mathbb{Z}^{n}}\sup_{t\in[0,T]}|k|^{m}|\hat{f}(t,k)|.
\end{equation}
In practice we will be dealing with functions with zero mean,
so it will be equivalent for us to treat the norms as
\begin{equation}\nonumber
\|f\|_{Y^{m}}=\sum_{k\in\mathbb{Z}^{n}\setminus\{0\}}|k|^{m}|\hat{f}(k)|,
\end{equation}
\begin{equation}\nonumber
\|f\|_{\mathcal{Y}^{m}}=\sum_{k\in\mathbb{Z}^{n}\setminus\{0\}}\sup_{t\in[0,T]}|k|^{m}|\hat{f}(t,k)|.
\end{equation}
We note that the space $X^{-1}$ as used in  \cite{ALN4}, \cite{baePAMS}, \cite{leiLin} is equal to our space $Y^{-1}.$

Given $m\in\mathbb{R},$ we also have a related function space on space-time, $\mathcal{X}^{m}.$
We define the space $\mathcal{X}^{m}$ according to the norm
\begin{equation}\label{x2Definition}
\|f\|_{\mathcal{X}^{m}}=\int_{0}^{T}|\hat{f}(t,0)|\ dt + \sum_{k\in\mathbb{Z}^{n}_{*}}\int_{0}^{T}|k|^{m}|\hat{f}(t,k)| \ dt.
\end{equation}
If $f$ has zero mean for all times, then this becomes simply
\begin{equation}\label{x2DefinitionZeroMean}
\|f\|_{\mathcal{X}^{m}}= \sum_{k\in\mathbb{Z}^{n}_{*}}\int_{0}^{T}|k|^{m}|\hat{f}(t,k)| \ dt.
\end{equation}
In the results to follow, we will typically take $m=2$ or $m=4.$

We will consider two cases in what follows.  We first describe Case A.  In Case A, we assume that all $L_{i}<2\pi,$ and we take
$T=\infty.$  Because of the size of the $L_{i},$ we have $\sigma(\Delta^{2}+\Delta)(k)>0$ for all $k\in\mathbb{Z}^{n}_{*}.$
Then we have
\begin{equation}\nonumber
\sup_{t\in[0,T]}\sup_{k\in\mathbb{Z}^{n}_{*}}e^{-t\sigma(\Delta^{2}+\Delta)(k)}
=\sup_{t\in[0,\infty)}\sup_{k\in\mathbb{Z}^{n}_{*}}e^{-t\sigma(\Delta^{2}+\Delta)(k)}=1.
\end{equation}
In Case B, we let $T\in(0,\infty)$ be given, and we assume there exists at least one $i\in\{1,\ldots,n\}$ such that $L_{1}\geq2\pi.$
Then there exists $M_{1}>0$ such that
\begin{equation}\label{M1}
\sup_{t\in[0,T]}\sup_{k\in\mathbb{Z}^{n}_{*}}e^{-t\sigma(\Delta^{2}+\Delta)(k)}\leq M_{1}.
\end{equation}
In Case B, we make the decomposition $\mathbb{Z}^{n}_{*}=\Omega_{F}\cup\Omega_{I},$
where for all $k\in\Omega_{F},$ the symbol is non-positive, i.e. $\sigma(\Delta^{2}+\Delta)(k)\leq0.$
Then on the complement, of course we have for all $k\in\Omega_{I},$ $\sigma(\Delta^{2}+\Delta)(k)>0.$
Of course the set $\Omega_{F}$ is finite and $\Omega_{I}$ is infinite.

We may of course also consider the decomposition $\mathbb{Z}^{n}_{*}=\Omega_{F}\cup\Omega_{I}$ in Case A as well,
and then we simply have $\Omega_{F}=\emptyset.$
In either case, we have there exists $M_{2}>0$ such that
\begin{equation}\label{M2}
\sigma(\Delta^{2}+\Delta)(k)>M_{2}|k|^{4},\qquad \forall k\in\Omega_{I}.
\end{equation}
We also introduce $M_{3}$ to be the maximum value of $|k|$ for $k\in\Omega_{F},$
\begin{equation}\label{M3}
|k|\leq M_{3},\qquad \forall k\in\Omega_{F}.
\end{equation}

We will rely on the following classical abstract result:
\begin{lemma} \label{fixedpoint}
Let ($X,$ $\vertiii{\cdot}_{X}$) be a Banach space. Assume that $\mathcal{B}:X \times X \to X$ is a continuous bilinear operator and let $\eta>0$ satisfy $\eta\geq \|\mathcal{B}\|_{X\times X\rightarrow X}$. Then, for any $x_0 \in X$ such that
\[4\eta \vertiii{x_0}_{X}<1,\]
there exists one and only one solution to the equation
\[x=x_0+\mathcal{B}(x,x) \qquad \text{ with } \vertiii{x}_{X} < \frac{1}{2\eta}.\]
Moreover, $\vertiii{x}_{X} \leq 2\vertiii{x_0}_{X}$.
\end{lemma}
See \cite[p. 37, Lemma 1.2.6]{Cannone1995} and \cite{AuscherTchamitchian1999,Cannone2003}.

We may write the mild formulation of the Kuramoto-Sivashinsky equation \eqref{KSEquationPsi} as
\begin{equation}\label{mildKS}
\psi=S\psi_{0}-\frac{1}{2}B(\psi,\psi).
\end{equation}
Here, the semigroup operator is
\begin{equation}\label{semigroup}
S\psi_{0}=e^{-t\sigma(\Delta^{2}+\Delta)(k)}\psi_{0},
\end{equation}
and the bilinear term is
\begin{equation}\label{bilinop}
B(F,G)=\int_{0}^{t}e^{-(t-s)(\Delta^{2}+\Delta)}\mathbb{P}(\nabla F\cdot\nabla G)\ ds.
\end{equation}
The Fourier coefficients of $B(F,G)$ are
\begin{multline}\label{bilinearTransform}
\widehat{B(F,G)}(t,k)=\int_{0}^{t}e^{-(t-s)\sigma(\Delta^{2}+\Delta)(k)}\mathcal{F}[\mathbb{P}(\nabla F\cdot\nabla G)](s,k)\ ds
\\
= \int_{0}^{t}e^{-(t-s)\sigma(\Delta^{2}+\Delta)(k)}\sum_{j\in\mathbb{Z}^{n}_{*}, j\neq k}\sum_{\ell=1}^{n}
\frac{2\pi i}{L_{\ell}}(k_{\ell}-j_{\ell})\hat{F}(s,k-j)\frac{2\pi i}{L_{\ell}} j_{\ell}\hat{G}(s,j)\ ds
\\
=
-\int_{0}^{t}e^{-(t-s)\sigma(\Delta^{2}+\Delta)(k)}\sum_{j\in\mathbb{Z}^{n}_{*}, j\neq k}\sum_{i=1}^{n}
\frac{4\pi^2}{L_{i}^2}(k_{i}-j_{i})\hat{F}(s,k-j)j_{i}\hat{G}(s,j)\ ds.
\end{multline}

In all of the estimates we will perform we will use only bounds from above, with respect to the frequency variable $k$, of $|\widehat{B(F,G)}(t,k)|$. We thus ignore, hereafter, the constants $\displaystyle{\frac{4\pi^2}{L_{i}^2}}$, absorbing them into a positive constant $C$, which is then normalized to $1$.

\section{Existence of solutions with data in $Y^{-1}$}\label{wienerDataBilinearSection}

In this section, we will prove the following theorem, giving existence of solutions with initial data taken from the
space $Y^{-1}.$

\begin{theorem} \label{existY-1} Let $T>0$ be given.  (If the conditions of Case A hold, then $T$ may be
taken to be $T=\infty.$)
Let $n\geq1.$  There exists $\varepsilon>0$ such that for any $\phi_{0}$ with $\mathbb{P}\phi_{0}\in Y^{-1},$ if
$\|\mathbb{P}\phi_{0}\|_{Y^{-1}}<\varepsilon,$ then there exists $\phi$ with 
$\mathbb{P}\phi\in\mathcal{Y}^{-1}\cap \mathcal{X}^{3}$ such that $\phi$ is a mild solution
to the initial value problem \eqref{KSEquationOriginal}, \eqref{data}.
\end{theorem}
\begin{proof}
To use Lemma \ref{fixedpoint}, we need to establish the bilinear estimate, and also that
$x_{0}=S\mathbb{P}\phi_{0}\in\mathcal{Y}^{-1}\cap\mathcal{X}^{3}.$

For the semigroup, we let $\psi_{0}\in Y^{-1}$ be given and we must show $S\psi_{0}\in\mathcal{Y}^{-1}\cap\mathcal{X}^{3}.$
We begin by computing the norm in $\mathcal{Y}^{-1}:$
\begin{equation}\nonumber
\|S\psi_{0}\|_{\mathcal{Y}^{-1}}=\sum_{k\in\mathbb{Z}^{n}_{*}}\sup_{t\in[0,T]}\frac{e^{-t\sigma(\Delta^{2}+\Delta)(k)}}{|k|}
|\hat{\psi}_{0}(k)|
\leq M_{1}\|\psi_{0}\|_{Y^{-1}}.
\end{equation}
We next compute the norm in $\mathcal{X}^{3}:$
\begin{multline}\nonumber
\|S\psi_{0}\|_{\mathcal{X}^{3}}=\sum_{k\in\mathbb{Z}^{n}_{*}}\int_{0}^{T}|k|^{3}
e^{-t\sigma(\Delta^{2}+\Delta)(k)}|\hat{\psi}_{0}(k)|\ dt
\\
\leq\left(\sum_{k\in\mathbb{Z}^{n}_{*}}\frac{|\hat{\psi}_{0}(k)|}{|k|}\right)
\sup_{k\in\mathbb{Z}^{n}_{*}}\left(\frac{|k|^{4}(1-e^{-T\sigma(\Delta^{2}+\Delta)(k)})}{\sigma(\Delta^{2}+\Delta)(k)}\right).
\end{multline}
In Case A, the supremum is finite because we may neglect the exponential and $\sigma(\Delta^{2}+\Delta)(k)>M_{2}|k|^{4}.$
In Case B, we may take the supremum separately over the sets $\Omega_{F}$ and $\Omega_{I},$ and the reasoning from
Case A applies to the supremum over $\Omega_{I}.$  For the supremum over $\Omega_{F},$ we find that it is finite because
$k$ is in a bounded set and $T$ is finite.  In either case, we have concluded that there exists $C>0$ such that $S\psi_{0}\in\mathcal{X}^{3}$ and
\begin{equation}\label{Spsi0estY-1}
  \|S\psi_{0}\|_{\mathcal{X}^{3} } \leq C \| \psi_0\|_{Y^{-1}}.
\end{equation}
This completes the proof of the needed semigroup properties.

We next need to compute $\|B(F,G)\|_{\mathcal{Y}^{-1}}$ and $\|B(F,G)\|_{\mathcal{X}^{3}}.$
We begin to compute the norm in $\mathcal{Y}^{-1}:$
\begin{equation}\nonumber
\|B(F,G)\|_{\mathcal{Y}^{-1}}=\sum_{k\in\mathbb{Z}^{n}_{*}}\sup_{t\in[0,T]}\frac{|\widehat{B(F,G)}(t,k)|}{|k|}.
\end{equation}
We substitute from \eqref{bilinearTransform}, and make some elementary bounds, arriving at
\begin{equation}\nonumber
\|B(F,G)\|_{\mathcal{Y}^{-1}}\leq nM_{1}\sum_{k\in\mathbb{Z}^{n}_{*}}
\int_{0}^{T}
\sum_{j\in\mathbb{Z}^{n}_{*}}\frac{1}{|k|}|k-j||\hat{F}(s,k-j)|j||\hat{G}(s,j)|\ ds.
\end{equation}
As in \cite{leiLin}, we use the inequality
\begin{equation}\label{simpleInequality}
1\leq \frac{|k-j|}{|j|}+\frac{|j|}{|k-j|}.
\end{equation}
Using this, we find
\begin{multline}\nonumber
\|B(F,G)\|_{\mathcal{Y}^{-1}}\leq nM_{1}\sum_{k\in\mathbb{Z}^{n}_{*}}\int_{0}^{T}\sum_{j\in\mathbb{Z}^{n}_{*}}
\frac{|k-j|^{2}}{|k|}|\hat{F}(s,k-j)||\hat{G}(s,j)|\ ds
\\
+nM_{1}\sum_{k\in\mathbb{Z}^{n}_{*}}\int_{0}^{T}\sum_{j\in\mathbb{Z}^{n}_{*}}
|\hat{F}(s,k-j)| \frac{|j|^{2}}{|k|}|\hat{G}(s,j)|\ ds.
\end{multline}
We then multiply and divide both terms by $|k-j||j|,$ arriving at
\begin{multline}\nonumber
\|B(F,G)\|_{\mathcal{Y}^{-1}}\leq nM_{1}\sum_{k\in\mathbb{Z}^{n}_{*}}\int_{0}^{T}\sum_{j\in\mathbb{Z}^{n}_{*}}
|k-j|^{3}|\hat{F}(s,k-j)|\frac{|\hat{G}(s,j)|}{|j|}\left(\frac{|j|}{|k||k-j|}\right)\ ds
\\
+nM_{1}\sum_{k\in\mathbb{Z}^{n}_{*}}\int_{0}^{T}\sum_{j\in\mathbb{Z}^{n}_{*}}
\frac{|\hat{F}(s,k-j)|}{|k-j|} |j|^{3}|\hat{G}(s,j)|\left(\frac{|k-j|}{|k||j|}\right)\ ds.
\end{multline}
We note the elementary bounds
\begin{equation}\label{simpleKJBound}
\frac{|j|}{|k||k-j|}\leq\frac{|k|+|k-j|}{|k||k-j|}\leq 2,\qquad
\frac{|k-j|}{|k||j|}\leq\frac{|k|+|j|}{|k||j|}\leq 2.
\end{equation}

This, then, immediately yields the bound
\begin{multline}\label{BFGY-1}
\|B(F,G)\|_{\mathcal{Y}^{-1}}\leq 2nM_{1}\|F\|_{\mathcal{X}^{3}}\|G\|_{\mathcal{Y}^{-1}}
+2nM_{1}\|F\|_{\mathcal{Y}^{-1}}\|G\|_{\mathcal{X}^{3}}.
\\
\leq 2nM_{1}(\|F\|_{\mathcal{Y}^{-1}}+\|F\|_{\mathcal{X}^{3}})
(\|G\|_{\mathcal{Y}^{-1}}+\|G\|_{\mathcal{X}^{3}}).
\end{multline}

We next will consider the higher norm of $B(F,G),$ attempting to bound $\|B(F,G)\|_{\mathcal{X}^{3}}.$
To begin we have
\begin{multline}\nonumber
\|B(F,G)\|_{\mathcal{X}^{3}}=\sum_{k\in\mathbb{Z}^{n}_{*}}\int_{0}^{T}|k|^{3}|\widehat{B(F,G)}(k)|\ dt =
\\
\sum_{k\in\mathbb{Z}^{n}_{*}}\int_{0}^{T}\left|\int_{0}^{t}|k|^{3}e^{-(t-s)\sigma(\Delta^{2}+\Delta)(k)}
\sum_{j\in\mathbb{Z}^{n}_{*}}\sum_{i=1}^{n}(k_{i}-j_{i})\hat{F}(s,k-j)j_{i}\hat{G}(s,j)\ ds\right|\ dt.
\end{multline}
We use the triangle inequality and \eqref{simpleInequality}, finding
\begin{multline}\nonumber
\|B(F,G)\|_{\mathcal{X}^{3}}\leq
\\
\sum_{k\in\mathbb{Z}^{n}_{*}}\int_{0}^{T}\int_{0}^{t}|k|^{3}e^{-(t-s)\sigma(\Delta^{2}+\Delta)(k)}\sum_{j\in\mathbb{Z}^{n}_{*}}
|k-j|^{2} |\hat{F}(s,k-j)| |\hat{G}(s,j)|\ ds dt
\\
+
\sum_{k\in\mathbb{Z}^{n}_{*}}\int_{0}^{T}\int_{0}^{t}|k|^{3}e^{-(t-s)\sigma(\Delta^{2}+\Delta)(k)}\sum_{j\in\mathbb{Z}^{n}_{*}}
|\hat{F}(s,k-j)| |j|^{2} |\hat{G}(s,j)|\ ds dt
\\
=A_{1}+A_{2}.
\end{multline}
We will only include the details for the estimate of $A_{1},$ as the estimate of $A_{2}$ is exactly the same.
We decompose $A_{1}$ further, using the decomposition $\mathbb{Z}^{n}_{*}=\Omega_{F}\cup\Omega_{I}.$  We have
\begin{multline}\nonumber
A_{1}=\sum_{k\in\Omega_{F}}\int_{0}^{T}\int_{0}^{t}|k|^{3}e^{-(t-s)\sigma(\Delta^{2}+\Delta)(k)}\sum_{j\in\mathbb{Z}^{n}_{*}}
|k-j|^{2} |\hat{F}(s,k-j)| |\hat{G}(s,j)|\ ds dt
\\
+ \sum_{k\in\Omega_{I}}\int_{0}^{T}\int_{0}^{t}|k|^{3}e^{-(t-s)\sigma(\Delta^{2}+\Delta)(k)}\sum_{j\in\mathbb{Z}^{n}_{*}}
|k-j|^{2} |\hat{F}(s,k-j)| |\hat{G}(s,j)|\ ds dt
\\
=A_{3}+A_{4}.
\end{multline}

We use the definitions of $M_{1}$, \eqref{M1}, and $M_{3}$, \eqref{M3}, to immediately bound $A_{3}$ as
\begin{equation}\nonumber
A_{3}\leq M_{1}M_{3}^{3}T\int_{0}^{T}\sum_{k\in\mathbb{Z}^{n}_{*}}\sum_{j\in\mathbb{Z}^{n}_{*}}
|k-j|^{2} |\hat{F}(s,k-j)| \left[\sup_{\tau\in[0,T]}|\hat{G}(\tau,j)|\right]\ ds.
\end{equation}
We then multiply and divide by $|k-j||j|$ and rearrange,  finding
\begin{multline}\nonumber
A_{3}\leq
M_{1}M_{3}^{3}(M_{3}+1) T\int_{0}^{T}\sum_{k\in\mathbb{Z}^{n}_{*}}\sum_{j\in\mathbb{Z}^{n}_{*}}
|k-j|^{3} |\hat{F}(s,k-j)| \left[\sup_{\tau\in[0,T]}\frac{|\hat{G}(\tau,j)}{|j|}\right]\ ds.
\\
\leq M_{1}M_{3}^{4}(M_{3}+1)T\|F\|_{\mathcal{X}^{3}}\|G\|_{\mathcal{Y}^{-1}}.
\end{multline}
Here we have used the definition of $M_{3}$ to make the elementary bound
\begin{equation}\nonumber
\sup_{k\in\Omega_{F}}\sup_{j\in\mathbb{Z}^{n},\ j\neq k}\frac{|j|}{|k-j|}
= \sup_{k\in\Omega_{F}}\sup_{\ell\in\mathbb{Z}^{n},\ell\neq0} \frac{|\ell+k|}{|\ell|}\leq
M_{3}+1.
\end{equation}

We next bound $A_{4}.$  We exchange the order of integration and compute the integral with respect to $t,$ finding
\begin{multline}\nonumber
A_{4}=\sum_{k\in\Omega_{I}}\int_{0}^{T}\int_{s}^{T}|k|^{3}e^{-(t-s)\sigma(\Delta^{2}+\Delta)(k)}
\sum_{j\in\mathbb{Z}^{n}_{*}}|k-j|^{2} |\hat{F}(s,k-j)| |\hat{G}(s,j)|\ dt ds
\\
=\sum_{k\in\Omega_{I}}\int_{0}^{T}\frac{|k|^{4}}{|k|}
\frac{1-e^{-(T-s)\sigma(\Delta^{2}+\Delta)(k)}}{\sigma(\Delta^{2}+\Delta)(k)}
\sum_{j\in\mathbb{Z}^{n}_{*}}|k-j|^{2} |\hat{F}(s,k-j)| |\hat{G}(s,j)|\ ds.
\end{multline}
Since $\sigma(\Delta^{2}+\Delta)(k)$ is positive for $k\in\Omega_{I},$ we may neglect the exponential in the numerator,
and use the definition of $M_{2},$ finding
\begin{equation}\nonumber
A_{4}\leq \frac{1}{M_{2}}\int_{0}^{T}\sum_{k\in\mathbb{Z}^{n}_{*}}\sum_{j\in\mathbb{Z}^{n}_{*}}
\frac{|k-j|^{2}}{|k|} |\hat{F}(s,k-j)| |\hat{G}(s,j)| \ ds.
\end{equation}
We then multiply and divide by $|k-j||j|,$ and again use \eqref{simpleKJBound}, finding
\begin{multline}\nonumber
A_{4}\leq\frac{1}{M_{2}}\int_{0}^{T}\sum_{k\in\mathbb{Z}^{n}_{*}}\sum_{j\in\mathbb{Z}^{n}_{*}}
|k-j|^{3} |\hat{F}(s,k-j)| \frac{|\hat{G}(s,j)|}{|j|}\left(\frac{|j|}{|k||k-j|}\right) \ ds
\\
\leq \frac{1}{M_{2}}\|F\|_{\mathcal{X}^{3}}\|G\|_{\mathcal{Y}^{-1}}.
\end{multline}

We have concluded the bound
\begin{equation}\nonumber
A_{1}\leq\left(M_{1}M_{3}^{3}(M_{3}+1)T+\frac{1}{M_{2}}\right)\|F\|_{\mathcal{X}^{3}}\|G\|_{\mathcal{Y}^{-1}},
\end{equation}
and we have by symmetry the corresponding estimate for $A_{2},$ namely
\begin{equation}\nonumber
A_{2}\leq\left(M_{1}M_{3}^{3}(M_{3}+1)T+\frac{1}{M_{2}}\right)\|F\|_{\mathcal{Y}^{-1}}\|G\|_{\mathcal{X}^{3}}.
\end{equation}
These bounds immediately imply the desired conclusion, which is
\begin{equation}\label{BFGestY-1X3}
\|B(F,G)\|_{\mathcal{X}^{3}}\leq \left(M_{1}M_{3}^{3}(M_{3}+1)T
+\frac{1}{M_{2}}\right)(\|F\|_{\mathcal{Y}^{-1}}+\|F\|_{\mathcal{X}^{3}})
(\|G\|_{\mathcal{Y}^{-1}}+\|G\|_{\mathcal{X}^{3}}).
\end{equation}
(It is understood that if the set $\Omega_{F}$ is empty then we may take $T=\infty,$ and that in this case
the combination $M_{1}T$ is understood as $M_{1}T=0.$)

\end{proof}

\section{Existence of solutions with data in pseudomeasure spaces}\label{pseudomeasureSection}

We prove existence theorems for the Kuramoto-Sivashinsky equation with pseudomeasure data in dimensions
$n=1$ and $n=2.$  We first define the pseudomeasure spaces.

For any $m\in\mathbb{R},$ we define the sets $PM^{m}$ and $\mathcal{PM}^{m}$ by their norms,
\begin{equation}\nonumber
\|f\|_{PM^{m}}=|\hat{f}(0)|+\sup_{k\in\mathbb{Z}^{n}}|k|^{m}|\hat{f}(k)|,
\end{equation}
\begin{equation}\nonumber
\|f\|_{\mathcal{PM}^{m}}=\sup_{t\in[0,T]}|\hat{f}(t,0)|+\sup_{t\in[0,T]}\sup_{k\in\mathbb{Z}^{n}}|k|^{m}|\hat{f}(t,k)|.
\end{equation}
As before, we will hereafter assume that all functions considered have zero mean, and will therefore only
need to use these norms on spaces of functions with zero mean.  We then get the simpler expressions
\begin{equation}\nonumber
\|f\|_{PM^{m}}=\sup_{k\in\mathbb{Z}^{n}_{*}}|k|^{m}|\hat{f}(k)|,
\end{equation}
\begin{equation}\nonumber
\|f\|_{\mathcal{PM}^{m}}=\sup_{t\in[0,T]}\sup_{k\in\mathbb{Z}^{n}_{*}}|k|^{m}|\hat{f}(t,k)|.
\end{equation}

In Section \ref{pseudomeasureLinear} we give the linear estimates, which are relevant for both dimensions $n=1$
and $n=2.$  In Section \ref{1DSection} we state the existence theorem in dimension $n=1$ and demonstrate
the needed bilinear estimates.  Then in Section \ref{2DSection} we state the existence theorem for
dimension $n=2$ and again give the needed bilinear estimates.

\subsection{Linear estimates}\label{pseudomeasureLinear}

We give the linear estimates for pseudomeasure data in the following two lemmas.

\begin{lemma}\label{firstSemigroupLemma}
For any $m\in\mathbb{R},$ the semigroup operator $S$ satisfies
$S:PM^{m}\rightarrow\mathcal{PM}^{m},$ with the estimate
\begin{equation}\nonumber
\|S\psi_{0}\|_{\mathcal{PM}^{m}}\leq M_{1}\|\psi_{0}\|_{{PM}^{m}}.
\end{equation}
\end{lemma}
\begin{proof}
We begin to estimate $\|S\psi_{0}\|_{\mathcal{PM}^{m}}.$  We have
\begin{multline}\nonumber
\|S\psi_{0}\|_{\mathcal{PM}^{m}}=\sup_{t\in[0,T]}\sup_{k\in\mathbb{Z}^{n}_{*}}|k|^{m}e^{-t\sigma(\Delta^{2}+\Delta)(k)}
|\hat{\psi}_{0}(k)|
\\
\leq \|\psi_{0}\|_{PM^{m}}\sup_{t\in[0,T]}\sup_{k\in\mathbb{Z}^{N}_{*}}e^{-t\sigma(\Delta^{2}+\Delta)(k)}.
\end{multline}
Using the definition of $M_{1},$ then, this is
\begin{equation}\nonumber
\|S\psi_{0}\|_{\mathcal{PM}^{m}}\leq M_{1}\|\psi_{0}\|_{PM^{m}}.
\end{equation}
\end{proof}

We next estimate $\|S\psi_{0}\|_{\mathcal{X}^{m_{2}}},$ with $\psi_{0}\in PM^{m_{1}}.$
\begin{lemma}\label{secondSemigroupLemma}
Let $m_{1}$ and $m_{2}$ be real numbers satisfying
$m_{2}-m_{1}-4<-n.$  There exists $K>0$ such that for any $\psi_{0}\in PM^{m-1},$ we have
$S\psi_{0}\in\mathcal{X}^{m_{2}}$
with the estimate
\begin{equation}\nonumber
\|S\psi_{0}\|_{\mathcal{X}^{m_{2}}}\leq K\|\psi_{0}\|_{PM^{m_{1}}}.
\end{equation}
\end{lemma}
\begin{proof} We begin by writing out the norm of $S\psi_{0}$ in the space $\mathcal{X}^{m_{2}}:$
\begin{equation}\nonumber
\|S\psi_{0}\|_{\mathcal{X}^{m_{2}}}=\int_{0}^{T}\sum_{k\in\mathbb{Z}^{n}_{*}}|k|^{m_{2}}e^{-t\sigma(\Delta^{2}+\Delta)(k)}
|\hat{\psi}_{0}(k)|\ dt.
\end{equation}
We multiply and divide the integrand by $|k|^{m_{1}},$ and we take out the norm of $\psi_{0},$ finding
\begin{equation}\nonumber
\|S\psi_{0}\|_{\mathcal{X}^{m_{2}}}\leq\|\psi_{0}\|_{PM^{m_{1}}}\int_{0}^{T}\sum_{k\in\mathbb{Z}^{n}_{*}}
|k|^{m_{2}-m_{1}}e^{-t(\sigma(\Delta^{2}+\Delta)(k))}\ dt.
\end{equation}
We then decompose $\mathbb{Z}^{n}_{*}$ into $\Omega_{F}\cup\Omega_{I},$ and we estimate the portion over the finite set
$\Omega_{F}.$  This yields
\begin{equation}\nonumber
\|S\psi_{0}\|_{\mathcal{X}^{m_{2}}}\leq\|\psi_{0}\|_{PM^{m_{1}}}\left(M_{1}C(\Omega_{F})T+\int_{0}^{T}\sum_{k\in\Omega_{I}}
|k|^{m_{2}-m_{1}}e^{-t\sigma(\Delta^{2}+\Delta)(k)}\ dt\right).
\end{equation}
(If we are in Case A, then since $\Omega_{F}=\emptyset,$ then in this case with $T=\infty$ we may take $C(\Omega_{F})=0,$
with the understanding that this would mean that $M_{1}C(\Omega_{F})T=0$ then.)
Then we evaluate the remaining integral, finding
\begin{equation}\nonumber
\|S\psi_{0}\|_{\mathcal{X}_{m_{2}}}\leq\|\psi_{0}\|_{PM^{m_{1}}}\left(M_{1}C(\Omega_{F})T+\sum_{k\in\Omega_{I}}
\frac{|k|^{m_{2}-m_{1}}(1-e^{-T(\sigma(\Delta^{2}+\Delta)(k))}}{\sigma(\Delta^{2}+\Delta)(k)}\right).
\end{equation}
We may then neglect the exponential and use \eqref{M2} to find
\begin{equation}\nonumber
\|S\psi_{0}\|_{\mathcal{X}^{m_{2}}}\leq\|\psi_{0}\|_{PM^{m_{1}}}\left(M_{1}C(\Omega_{F})T+\frac{1}{M_{2}}
\sum_{k\in\Omega_{I}}|k|^{m_{2}-m_{1}-4}\right).
\end{equation}
The series on the right-hand side converges, so we have concluded that there exists $K>0$ such that
\begin{equation}\nonumber
\|S\psi_{0}\|_{\mathcal{X}^{m_{2}}}\leq K\|\psi_{0}\|_{PM^{m_{1}}}.
\end{equation}
\end{proof}

\begin{remark}\label{mRemark}
The quantity $m_{2}-m_{1}$ describes how many derivatives the solution gains at positive times, compared to the data.
Of course one may expect, in $L^{1}$-based spaces, to gain four derivatives from a fourth-order parabolic evolution; at the same
time, because our nonlinearity only contains first derivatives, we do not require this full gain of four derivatives.
Our requirement $m_{2}-m_{1}-4<n$ implies that in one space dimension we may take $m_{2}-m_{1}<3,$ and in two space
dimensions we may take $m_{2}-m_{1}<2.$  In these cases this is sufficient gain of regularity to
establish the bilinear estimates.
\end{remark}

\subsection{Existence of solutions with pseudomeasure data with $n=1$}\label{1DSection}

In one space dimension, we can find the existence of solutions with $PM^{-p}$ data, for any $p\in(0,1/2).$
With the parabolic gain of regularity, we will also have that the solutions are in $\mathcal{X}^{2+p};$ this gain of
$2+2p$ derivatives is less than the four full derivatives which might be possible, but this is sufficient gain to deal with
the nonlinearity.
Note that these choices satisfy the constraints as discussed in Remark \ref{mRemark}.  Specifically, with
$m_{1}=-p$ and $m_{2}=2+p,$ and with $p<1/2,$ we have $m_{2}-m_{1}=2+2p<3,$ as desired.

\begin{theorem}\label{n=1PseudomeasureTheorem}
 Let $p\in(0,1/2)$  and  $T>0$ be given.
(If the conditions of Case A hold, then $T$ may be
taken to be $T=\infty.$)
Let $n=1.$  There exists $\varepsilon>0$ such that for any $\phi_{0}$ such that $\mathbb{P}\phi_{0}\in PM^{-p},$ if
$\|\mathbb{P}\phi_{0}\|_{PM^{-p}}<\varepsilon,$ then there exists $\phi$ with
$\mathbb{P}\phi\in\mathcal{PM}^{-p}\cap \mathcal{X}^{2+p}$
such that $\phi$ is a mild solution to the initial value problem \eqref{KSEquationOriginal}, \eqref{data}.
\end{theorem}

\begin{proof}  To apply
Lemma \ref{fixedpoint}, we need to conclude $x_{0}=S\mathbb{P}\phi_{0}$ is in the space
$\mathcal{PM}^{-p}\cap\mathcal{X}^{2+p}.$  This follows from Lemma \ref{firstSemigroupLemma} and also from
Lemma \ref{secondSemigroupLemma} with $m_{1}=-p$ and $m_{2}=2+p,$ since these parameters satisfy the condition
$m_{2}-m_{1}-4<-n.$

We begin with the
estimate of $B(F,G)$ in $\mathcal{PM}^{-p}.$
Using the definition of $\mathcal{PM}^{-p}$ and the triangle inequality, we have
\begin{multline}\nonumber
\|B(F,G)\|_{\mathcal{PM}^{-p}}
\\
=\sup_{t\in[0,T]}\sup_{k\in\mathbb{Z}_{*}}\left|\int_{0}^{t}\frac{1}{|k|^{p}}
e^{-(t-s)\sigma(\Delta^{2}+\Delta)(k)}\sum_{j\in\mathbb{Z}_{*}}j\hat{F}(s,j)(k-j)\hat{G}(s,k-j)\ ds\right|
\\
\leq \sup_{t\in[0,T]}\sup_{k\in\mathbb{Z}_{*}}\int_{0}^{t}\frac{e^{-(t-s)\sigma(\Delta^{2}+\Delta)(k)}}{|k|^{p}}
\sum_{j\in\mathbb{Z}_{*}}
|j| |\hat{F}(s,j)| |k-j| |\hat{G}(s,k-j)|\ ds.
\end{multline}
We then use Young's inequality on $j(k-j),$ and bound the exponentials by $M_{1}$ (recall the definition of $M_{1}$ in
\eqref{M1}), finding
\begin{multline}\nonumber
\|B(F,G)\|_{\mathcal{PM}^{-p}}
\\
\leq \sup_{t\in[0,T]}\sup_{k\in\mathbb{Z}_{*}}\frac{M_{1}}{2}\int_{0}^{t}
\frac{1}{|k|^{p}}\sum_{j\in\mathbb{Z}_{*}}
|j|^{2} |\hat{F}(s,j)| |\hat{G}(s,k-j)|\ ds
\\
+\sup_{t\in[0,T]}\sup_{k\in\mathbb{Z}_{*}}\frac{M_{1}}{2}\int_{0}^{t}
\frac{1}{|k|^{p}}\sum_{j\in\mathbb{Z}_{*}}
 |\hat{F}(s,j)| |k-j|^{2}|\hat{G}(s,k-j)|\ ds.
\end{multline}
We multiply and divide by the appropriate powers of $|j|$ and $|k-j|:$
\begin{multline}\label{almostDonePM-p}
\|B(F,G)\|_{\mathcal{PM}^{-p}}
\\
\leq \sup_{t\in[0,T]}\sup_{k\in\mathbb{Z}_{*}}\frac{M_{1}}{2}\int_{0}^{t}
\sum_{j\in\mathbb{Z}_{*}}\left(\frac{|k-j|^{p}}{|k|^{p}|j|^{p}}\right)
|j|^{2+p} |\hat{F}(s,j)| \left(\frac{|\hat{G}(s,k-j)|}{|k-j|^{p}}\right)\ ds
\\
+\sup_{t\in[0,T]}\sup_{k\in\mathbb{Z}_{*}}\frac{M_{1}}{2}\int_{0}^{t}
\sum_{j\in\mathbb{Z}_{*}}\left(\frac{|j|^{p}}{|k|^{p}|k-j|^{p}}\right)
 \left(\frac{|\hat{F}(s,j)|}{|j|^{p}}\right) |k-j|^{2+p}|\hat{G}(s,k-j)|\ ds.
\end{multline}
We have the elementary bounds
\begin{equation}\label{elementaryBoundsThreeFactors}
\frac{|k-j|^{p}}{|k|^{p}|j|^{p}}\leq 2^{p}, \qquad \frac{|j|^{p}}{|k|^{p}|k-j|^{p}}\leq 2^{p}.
\end{equation}
For the first term on the right-hand side of \eqref{almostDonePM-p},
we take the supremum with respect to $s$ and $k$ for the factor
$\frac{|\hat{G}(s,k-j)|}{|k-j|^{p}}$ in the integrand.  For the second term on the right-hand side, we take the supremum with respect to $s$ and $j$ for the factor
$\hat{F}(s,j)$ in the integrand.  These considerations lead to the following bound:
\begin{equation}\nonumber
\|B(F,G)\|_{\mathcal{PM}^{-p}}
\\
\leq 2^{p-1}M_{1}
\left(\|G\|_{\mathcal{PM}^{-p}}\|F\|_{\mathcal{X}^{2+p}}
+
\|F\|_{\mathcal{PM}^{-p}}\|G\|_{\mathcal{X}^{2+p}}\right)
\end{equation}
We may further bound this as
\begin{equation}\label{BFGPM-p}
\|B(F,G)\|_{\mathcal{PM}^{-p}}\leq 2^{p-1}M_{1}(\|F\|_{\mathcal{PM}^{-p}}+\|F\|_{\mathcal{X}^{2+p}})
(\|G\|_{\mathcal{PM}^{-p}}+\|G\|_{\mathcal{X}^{2+p}}).
\end{equation}

We next must estimate $B(F,G)$ in the space $\mathcal{X}^{2+p}.$  We begin with the definition, and use the triangle
inequality:
\begin{multline}\nonumber
\|B(F,G)\|_{\mathcal{X}^{2+p}}\leq
\\
\sum_{k\in\mathbb{Z}_{*}}\int_{0}^{T}|k|^{2+p} \int_{0}^{t}e^{-(t-s)\sigma(\Delta^{2}+\Delta)(k)}
\sum_{j\in\mathbb{Z}_{*}}|j| |\hat{F}(s,j)| |k-j| |\hat{G}(s,k-j)|\ ds dt.
\end{multline}
Then, as before, we use Young's inequality on $j(k-j),$
\begin{multline}\nonumber
\|B(F,G)\|_{\mathcal{X}^{2+p}}
\\
\leq\frac{1}{2}\sum_{k\in\mathbb{Z}_{*}}\int_{0}^{T}|k|^{2+p}\int_{0}^{t}e^{-(t-s)\sigma(\Delta^{2}+\Delta)(k)}
\sum_{j\in\mathbb{Z}_{*}}
|j|^{2}|\hat{F}(s,j)| |\hat{G}(s,k-j)|\ ds dt
\\
+\sum_{k\in\mathbb{Z}_{*}}\frac{1}{2}\int_{0}^{T}|k|^{2+p}\int_{0}^{t}e^{-(t-s)\sigma(\Delta^{2}+\Delta)(k)}
\sum_{j\in\mathbb{Z}_{*}}
|k-j|^{2}|\hat{F}(s,j)| |\hat{G}(s,k-j)|\ ds dt.
\end{multline}
For the first term on the right-hand side, we find the $\mathcal{PM}^{-p}$-norm of $G$ by multiplying and dividing by
$|k-j|^{p}$ and taking a supremum, and we also multiply and divide by $|j|^{p}.$
We treat the second term on the right-hand side similarly, and we arrive at
\begin{multline}\label{almostX2+p}
\|B(F,G)\|_{\mathcal{X}^{2+p}}
\\
\leq
\frac{\|G\|_{\mathcal{PM}^{-p}}}{2}\sum_{k\in\mathbb{Z}_{*}}\int_{0}^{T}\int_{0}^{t}
|k|^{2+p}e^{-(t-s)\sigma(\Delta^{2}+\Delta)(k)}\sum_{j\in\mathbb{Z}_{*}}\frac{|k-j|^{p}|j|^{2+p}}{|j|^{p}} |F(s,j)|\ ds dt
\\
+\frac{\|F\|_{\mathcal{PM}^{-p}}}{2}\sum_{k\in\mathbb{Z}_{*}}\int_{0}^{T}\int_{0}^{t}
|k|^{2+p}e^{-(t-s)\sigma(\Delta^{2}+\Delta)(k)}\sum_{j\in\mathbb{Z}_{*}}\frac{|k-j|^{2+p}|j|^{p}}{|k-j|^{p}} |G(s,k-j)|\ ds dt.
\end{multline}
We may then make the elementary estimates
\begin{equation}\label{elementaryEstimateTwoFactors}
\frac{|k-j|^{p}}{|j|^{p}}\leq \left(\frac{|k|+|j|}{|j|}\right)^{p}\leq 2^{p}|k|^{p},
\qquad
\frac{|j|^{p}}{|k-j|^{p}}
\leq \left(\frac{|k-j|+|k|}{|k-j|}\right)^{p}\leq 2^{p}|k|^{p}.
\end{equation}
Using these estimates with \eqref{almostX2+p}, we have
\begin{multline}\nonumber
\|B(F,G)\|_{\mathcal{X}^{2+p}}
\\
\leq
2^{p-1}\|G\|_{\mathcal{PM}^{-p}}\sum_{k\in\mathbb{Z}_{*}}\int_{0}^{T}\int_{0}^{t}
|k|^{2+2p}e^{-(t-s)\sigma(\Delta^{2}+\Delta)(k)}\sum_{j\in\mathbb{Z}_{*}}|j|^{2+p} |F(s,j)|\ ds dt
\\
+2^{p-1}\|F\|_{\mathcal{PM}^{-p}}\sum_{k\in\mathbb{Z}_{*}}\int_{0}^{T}\int_{0}^{t}
|k|^{2+2p}e^{-(t-s)\sigma(\Delta^{2}+\Delta)(k)}\sum_{j\in\mathbb{Z}_{*}}|k-j|^{2+p} |G(s,k-j)|\ ds dt.
\end{multline}
In the second term on the right-hand side we change the variable in the final summation, and
we also change the order of integration in both terms on the right-hand side, finding
\begin{multline}\label{almostDone1D}
\|B(F,G)\|_{\mathcal{X}^{2+p}}
\\
\leq 2^{p-1}\|G\|_{\mathcal{PM}^{-p}}\int_{0}^{T}
\left(\sum_{j\in\mathbb{Z}_{*}}|j|^{2+p}|\hat{F}(s,j)|\right)
\left(\sum_{k\in\mathbb{Z}_{*}}|k|^{2+2p}\int_{s}^{T}e^{-(t-s)\sigma(\Delta^{2}+\Delta)(k)}\ dt\right)\ ds.
\\
+2^{p-1}\|F\|_{\mathcal{PM}^{-p}}\int_{0}^{T}
\left(\sum_{j\in\mathbb{Z}_{*}}|j|^{2+p}|\hat{G}(s,j)|\right)
\left(\sum_{k\in\mathbb{Z}_{*}}|k|^{2+2p}\int_{s}^{T}e^{-(t-s)\sigma(\Delta^{2}+\Delta)(k)}\ dt\right)\ ds.
\end{multline}

We will work now with the sum with respect to $k,$ which is the same in both of the terms on the right-hand side.
We split it into the sum over $\Omega_{F}$ and the sum over $\Omega_{I}.$
Considering $k\in\Omega_{F},$ we have
\begin{equation}\nonumber
\sum_{k\in\Omega_{F}}k^{2+2p}\int_{s}^{T}e^{-(t-s)\sigma(\Delta^{2}+\Delta)(k)}\ dt\leq |\Omega_{F}|M_{1}M_{3}^{2+2p}T.
\end{equation}
(We have said that if $\Omega_{F}=\emptyset$ then we may take $T=\infty,$ and then this product is to be understood
as $|\Omega_{F}|T=0.$)
Considering $k\in\Omega_{I},$ we evaluate the integral, finding
\begin{equation}\nonumber
\sum_{k\in\Omega_{I}}k^{2+2p}\int_{s}^{T}e^{-(t-s)\sigma(\Delta^{2}+\Delta)(k)}\ dt
=\sum_{k\in\Omega_{I}}k^{2+2p}\left(\frac{1-e^{-(T-s)\sigma(\Delta^{2}+\Delta)(k)}}{\sigma(\Delta^{2}+\Delta)(k)}\right).
\end{equation}
Since the denominator is positive for $k\in\Omega_{I},$ we may neglect the exponential in the numerator.  Then we use the definition
of $M_{2},$ finding
\begin{equation}\nonumber
\sum_{k\in\Omega_{I}}k^{2+2p}\int_{s}^{T}e^{-(t-s)\sigma(\Delta^{2}+\Delta)(k)}\ dt \leq \sum_{k\in\mathbb{Z}_{*}}
\frac{M_{2}}{k^{2-2p}}=M_{2}c(p)<\infty.
\end{equation}
We of course have used here that $p<1/2.$

Returning to \eqref{almostDone1D}, we conclude with the bound
\begin{multline}\label{BFGX2andp1}
\|B(F,G)\|_{\mathcal{X}^{2+p}}\leq
2^{p-1}\left(|\Omega_{F}|M_{1}M_{3}^{2+2p}T+M_{2}c(p)\right)\cdot
\\
\cdot
(\|F\|_{\mathcal{PM}^{-p}}+\|F\|_{\mathcal{X}^{2+p}})
(\|G\|_{\mathcal{PM}^{-p}}+\|G\|_{\mathcal{X}^{2+p}}).
\end{multline}
\end{proof}

\subsection{Existence of solutions with pseudomeasure data with $n=2$}\label{2DSection}

We again let $p\in(0,1/2)$ be given.
In the case of two space dimensions, we will be taking data in $PM^{1-p},$ and finding solutions in
$\mathcal{PM}^{1-p}\cap\mathcal{X}^{2+p}.$  As regards Remark \ref{mRemark}, this means that we have $m_{1}=1-p$
and $m_{2}=2+p,$ so that $m_{2}-m_{1}=1+2p<2,$ as required.

\begin{theorem}\label{n=2PseudomeasureTheorem}
Let $p\in(0,1/2)$  and  $T>0$ be given.
(If the conditions of Case A hold, then $T$ may be
taken to be $T=\infty.$)
Let $n=2.$  There exists $\varepsilon>0$ such that for any $\phi_{0}$ with $\mathbb{P}\phi_{0}\in PM^{1-p},$ if
$\|\mathbb{P}\phi_{0}\|_{PM^{1-p}}<\varepsilon,$ then there exists $\phi$ with
$\mathbb{P}\phi\in\mathcal{PM}^{1-p}\cap \mathcal{X}^{2+p}$ such that $\phi$ is a mild solution
to the initial value problem \eqref{KSEquationOriginal}, \eqref{data}.
\end{theorem}

\begin{proof}  To apply
Lemma \ref{fixedpoint}, we need to conclude $x_{0}=S\mathbb{P}\phi_{0}$ is in the space
$\mathcal{PM}^{1-p}\cap\mathcal{X}^{2+p}.$  This follows from Lemma \ref{firstSemigroupLemma} and also from
Lemma \ref{secondSemigroupLemma} with $m_{1}=1-p$ and $m_{2}=2+p,$ since these parameters satisfy the condition
$m_{2}-m_{1}-4<-n.$

We estimate $\|B(F,G)\|_{\mathcal{PM}^{1-p}}.$  From the definition of the norm and $I^{+},$ we have
\begin{multline}\nonumber
\|B(F,G)\|_{\mathcal{PM}^{1-p}} =
\\
\sup_{t\in[0,T]}\sup_{k\in\mathbb{Z}^{2}_{*}}|k|^{1-p}\left|
\int_{0}^{t}e^{-(t-s)\sigma(\Delta^{2}+\Delta)(k)}\sum_{j\in\mathbb{Z}^{2}_{*}}\sum_{i=1}^{2}
(k_{i}-j_{i})\hat{F}(s,k-j)j_{i}\hat{G}(s,j)\ ds\right|.
\end{multline}
We use the triangle inequality, and the definition of the constant $M_{1},$ to find
\begin{equation}\nonumber
\|B(F,G)\|_{\mathcal{PM}^{1-p}}
\leq
2M_{1}\sup_{k\in\mathbb{Z}^{2}_{*}}|k|^{1-p}\int_{0}^{T}\sum_{j\in\mathbb{Z}^{2}_{*}}
|k-j| |\hat{F}(s,k-j)| |j| |\hat{G}(s,j)|\ ds.
\end{equation}
Bounding $|k|$ as $|k|\leq |k-j|+|j|,$ this becomes
\begin{multline}\nonumber
\|B(F,G)\|_{\mathcal{PM}^{1-p}}
\leq 2M_{1}\sup_{k\in\mathbb{Z}^{2}_{*}}\int_{0}^{T}\sum_{j\in\mathbb{Z}^{2}_{*}}
\frac{|k-j|^{2}|j|}{|k|^{p}} |\hat{F}(s,k-j)|  |\hat{G}(s,j)|\ ds
\\
+2M_{1}\sup_{k\in\mathbb{Z}^{2}_{*}}\int_{0}^{T}\sum_{j\in\mathbb{Z}^{2}_{*}}
\frac{ |k-j| |j|^{2}}{|k|^{p}} |\hat{F}(s,k-j)|  |\hat{G}(s,j)|\ ds.
\end{multline}
We adjust the factors of $|k-j|$ and $|j|,$ to find
\begin{multline}\nonumber
\|B(F,G)\|_{\mathcal{PM}^{1-p}}
\\
\leq 2M_{1}\sup_{k\in\mathbb{Z}^{2}_{*}}\int_{0}^{T}\sum_{j\in\mathbb{Z}^{2}_{*}}
\left(\frac{|j|^{p}}{|k|^{p}|k-j|^{p}}\right) |k-j|^{2+p}|\hat{F}(s,k-j)|  |j|^{1-p}|\hat{G}(s,j)|\ ds
\\
+2M_{1}\sup_{k\in\mathbb{Z}^{2}_{*}}\int_{0}^{T}\sum_{j\in\mathbb{Z}^{2}_{*}}
\left(\frac{ |k-j|^{p}}{|k|^{p}|j|^{p}}\right) |k-j|^{1-p}|\hat{F}(s,k-j)| |j|^{2+p} |\hat{G}(s,j)|\ ds.
\end{multline}
We again use \eqref{elementaryBoundsThreeFactors}, which is valid regardless of dimension.
In the first term on the right-hand side we take the supremum of $|j|^{1-p} |\hat{G}(s,j)|$ with respect to both $s$ and $j,$
and we treat the second term on the right-hand side similarly.  It is then immediate that
\begin{equation}\label{BFGPM1lessp}
\|B(F,G)\|_{\mathcal{PM}^{1-p}}\leq 2^{p+1}M_{1} (\|F\|_{\mathcal{PM}^{1-p}}+\|F\|_{\mathcal{X}^{2+p}})
(\|G\|_{\mathcal{PM}^{1-p}}+ \|G\|_{\mathcal{X}^{2+p}}).
\end{equation}

Next we bound $B(F,G)$ in $\mathcal{X}^{2+p}.$  From the definition of the norm and of $B(F,G),$ we have
\begin{multline}\nonumber
\|B(F,G)\|_{\mathcal{X}^{2+p}} =
\\
\int_{0}^{T}\sum_{k\in\mathbb{Z}^{2}_{*}}|k|^{2+p}\left|\int_{0}^{t}
e^{-(t-s)\sigma(\Delta^{2}+\Delta)(k)}\sum_{j\in\mathbb{Z}^{2}_{*}}
((k-j)\cdot j)\hat{F}(s,k-j)\hat{G}(s,j)\ ds\right|\ dt.
\\
\leq
2\int_{0}^{T}\sum_{k\in\mathbb{Z}^{2}_{*}}|k|^{2+p}\int_{0}^{t}e^{-(t-s)\sigma(\Delta^{2}+\Delta)(k)}
\sum_{j\in\mathbb{Z}^{2}_{*}}|k-j| |\hat{F}(s,k-j)| |j| |\hat{G}(s,j)| \ ds dt.
\end{multline}
We then write $|k|\leq |k-j|+|j|$ for just one factor of $|k|$ on the right-hand side:
\begin{multline}\nonumber
\|B(F,G)\|_{\mathcal{X}^{2+p}}
\\
\leq
2\int_{0}^{T}\sum_{k\in\mathbb{Z}^{2}_{*}}|k|^{1+p}
\int_{0}^{t}e^{-(t-s)\sigma(\Delta^{2}+\Delta)(k)}\sum_{j\in\mathbb{Z}^{2}_{*}}
|k-j|^{2} |\hat{F}(s,k-j)| |j| |\hat{G}(s,j)|
\ ds dt
\\
+
2\int_{0}^{T}\sum_{k\in\mathbb{Z}^{2}_{*}}|k|^{1+p}
\int_{0}^{t}e^{-(t-s)\sigma(\Delta^{2}+\Delta)(k)}\sum_{j\in\mathbb{Z}^{2}_{*}}
|k-j| |\hat{F}(s,k-j)| |j|^{2} |\hat{G}(s,j)|
\ ds dt.
\end{multline}
We next adjust factors of $|k-j|$ and $|j|,$ finding
\begin{multline}\nonumber
\|B(F,G)\|_{\mathcal{X}^{2+p}}
\\
\leq
2\int_{0}^{T}\sum_{k\in\mathbb{Z}^{2}_{*}}|k|^{1+p}
\int_{0}^{t}\Bigg[e^{-(t-s)\sigma(\Delta^{2}+\Delta)(k)}\cdot
\\
\cdot\sum_{j\in\mathbb{Z}^{2}_{*}}
\left(\frac{|j|^{p}}{|k-j|^{p}}\right)|k-j|^{2+p} |\hat{F}(s,k-j)| |j|^{1-p} |\hat{G}(s,j)|
\Bigg]\ ds dt
\\
+
2\int_{0}^{T}\sum_{k\in\mathbb{Z}^{2}_{*}}|k|^{1+p}
\int_{0}^{t}\Bigg[e^{-(t-s)\sigma(\Delta^{2}+\Delta)(k)}\cdot
\\
\cdot\sum_{j\in\mathbb{Z}^{2}_{*}}
\left(\frac{|k-j|^{p}}{|j|^{p}}\right)
|k-j|^{1-p} |\hat{F}(s,k-j)| |j|^{2+p} |\hat{G}(s,j)|
\Bigg]\ ds dt.
\end{multline}
We again use \eqref{elementaryEstimateTwoFactors}, which is valid regardless of dimension.
For the first term on the right-hand side, we bring out $\|G\|_{\mathcal{PM}^{1-p}}$ and change variables in the sum with respect to $j.$
In the second term on the right-hand side, we bring out $\|F\|_{\mathcal{PM}^{1-p}}.$  We also change the order of integration in both
of these terms, arriving at the bound
\begin{multline}\nonumber
\|B(F,G)\|_{\mathcal{X}^{2+p}}
\\
\leq
2^{p+1}\|G\|_{\mathcal{PM}^{1-p}}\int_{0}^{T}\left(\sum_{j\in\mathbb{Z}^{2}_{*}}|j|^{2+p}|F(s,j)|\right)
\left(\sum_{k\in\mathbb{Z}^{2}_{*}}\int_{s}^{T}|k|^{1+2p}e^{-(t-s)\sigma(\Delta^{2}+\Delta)(k)}\ dt\right)\ ds
\\
+
2^{p+1}\|F\|_{\mathcal{PM}^{1-p}}\int_{0}^{T}\left(\sum_{j\in\mathbb{Z}^{2}_{*}}|j|^{2+p}|G(s,j)|\right)
\left(\sum_{k\in\mathbb{Z}^{2}_{*}}\int_{s}^{T}|k|^{1+2p}e^{-(t-s)\sigma(\Delta^{2}+\Delta)(k)}\ dt\right)\ ds.
\end{multline}
We again work with the second factor of the integrand (which is the same in both of the terms on the right-hand side).  We decompose the sum over $k$ using
$\mathbb{Z}^{2}_{*}=\Omega_{F}\cup\Omega_{I}.$  Considering $k\in\Omega_{F},$ we have
\begin{equation}\nonumber
\sum_{k\in\Omega_{F}}\int_{s}^{T}|k|^{1+2p}e^{-(t-s)\sigma(\Delta^{2}+\Delta)(k)}\ dt
\leq |\Omega_{F}|M_{1}M_{3}^{1+2p}T.
\end{equation}
Considering next $k\in\Omega_{I},$ we evaluate the integral and find
\begin{equation}\nonumber
\sum_{k\in\Omega_{I}}\int_{s}^{T}|k|^{1+2p}e^{-(t-s)\sigma(\Delta^{2}+\Delta)(k)}\ dt
=\sum_{k\in\Omega_{I}}|k|^{1+2p}\frac{1-e^{-(T-s)\sigma(\Delta^{2}+\Delta(k)}}{\sigma(\Delta^{2}+\Delta)(k)}.
\end{equation}
Since the denominator on the right-hand side is positive, we may neglect the exponential in the numerator on the right-hand side,
and we may use the definition of $M_{2}$ to find
\begin{equation}\nonumber
\sum_{k\in\Omega_{I}}\int_{s}^{T}|k|^{1+2p}
e^{-(t-s)\sigma(\Delta^{2}+\Delta)(k)}\ dt \leq \frac{1}{M_{2}}\sum_{k\in\mathbb{Z}^{2}_{*}}
\frac{1}{|k|^{3-2p}}.
\end{equation}
Since $p<1/2,$ this sum converges.  We then conclude that
\begin{multline}\label{BFGX2andp2}
\|B(F,G)\|_{\mathcal{X}^{2+p}}\leq
2^{p+1}\left(|\Omega_{F}|M_{1}M_{3}^{1+2p}T+\frac{1}{M_{2}}\sum_{k\in\mathbb{Z}^{2}_{*}}\frac{1}{|k|^{3-2p}}\right)\cdot
\\
\cdot(\|F\|_{\mathcal{PM}^{1-p}}+\|F\|_{\mathcal{X}^{2+p}})
(\|G\|_{\mathcal{PM}^{1-p}}+\|G\|_{\mathcal{X}^{2+p}}).
\end{multline}
\end{proof}

\section{Analyticity}\label{analyticitySection}

In this section we will show that the solutions produced above are analytic within their time of existence, if needed by further restricting the size of the initial data.

Given initial data $\psi_0$ we recall the mild formulation of the Kuramoto-Sivashinsky equation \eqref{mildKS}:
\[\psi = S\psi_0 - \frac{1}{2} B(\psi,\psi),\]
where the semigroup $S$ was introduced in \eqref{semigroup} and the bilinear term $B=B(F,G)$ was given in \eqref{bilinop}.

Our approach to establish analyticity follows the one used by H. Bae in \cite{baePAMS}, in which one revisits the existence proofs but for an exponentially-weighted modification of $\psi$. More precisely, let $g=g(t)$ be a given function and consider
\begin{equation}\label{Vdef}
  V\equiv e^{g(t)|D|}\psi,
\end{equation}
where $|D|= \sqrt{-\Delta}$. Then $V$ should satisfy the equation
\begin{equation} \label{Videntweightg}
V = e^{[g(t)|D|-t(\Delta^2 + \Delta)]}V_0 - \frac{1}{2} \int_0^t e^{[g(t)|D|-(t-s)(\Delta^2 + \Delta)]}[\mathbb{P}(|\nabla e^{-g(s)|D|}V|^2)] \, ds,
\end{equation}
with $V_0 = \psi_0$. Existence of a solution to this equation for suitable $g$ and sufficiently small $V_0$ in certain function spaces then implies analyticity of $\psi$, as will be made precise at the end of this section. The radius of analyticity is bounded from below by $g(t)$.

We re-write \eqref{Videntweightg}, separating the linear from the nonlinear term, as
\begin{equation}\nonumber
  V = \mathcal{L}V_0 -\frac{1}{2}\mathcal{B}(V,V),
\end{equation}
with
\begin{equation}\label{Vlin}
  \calL V_0 = e^{[g(t)|D|-t(\Delta^2 + \Delta)]}V_0
\end{equation}
and
\begin{equation}\nonumber
\calB(U,W)= \int_0^t e^{[g(t)|D|-(t-s)(\Delta^2 + \Delta)]}[\mathbb{P}(\nabla e^{-g(s)|D|}U\cdot \nabla e^{-g(s)|D|}W)] \, ds.
\end{equation}

We will prove existence of a solution $V$  to \eqref{Videntweightg} for initial data $V_0 = \psi_0$ in $Y^{-1}$, in any dimension, and in $PM^{(n-1)-p}$, in dimensions $n=1$ and $n=2$, with $0<p<1/2$. As before, we use Lemma \ref{fixedpoint}, so we require the following bounds:
\begin{align}
\label{LY-1}   & \|\calL V_0\|_{\calY^{-1}}\leq C \|V_0\|_{Y^{-1}}\\
\label{LX3}    & \|\calL V_0\|_{\calX^{3}}\leq C \|V_0\|_{Y^{-1}}\\
\label{LPM}    & \|\calL V_0\|_{\calPM^{(n-1)-p}}\leq C \|V_0\|_{PM^{(n-1)-p}}\\
\label{LX2andp}& \|\calL V_0\|_{\calX^{2+p}}\leq C \|V_0\|_{PM^{(n-1)-p}}
\end{align}
as well as
\begin{align}
\label{BY-1}   & \|\calB(U,W)\|_{\calY^{-1}}\leq C (\|U\|_{\calY^{-1}} + \|U\|_{\calX^3})(\|W\|_{\calY^{-1}} + \|W\|_{\calX^3})\\
\label{BX3}    & \|\calB(U,W)\|_{\calX^{3}}\leq C (\|U\|_{\calY^{-1}} + \|U\|_{\calX^3})(\|W\|_{\calY^{-1}} + \|W\|_{\calX^3})\\
\label{BPM}    & \|\calB(U,W)\|_{\calPM^{(n-1)-p}}\leq C (\|U\|_{\calPM^{(n-1)-p}} + \|U\|_{\calX^{2+p}})(\|W\|_{\calPM^{(n-1)-p}} + \|W\|_{\calX^{2+p}})\\
\label{BX2andp}& \|\calB(U,W)\|_{\calX^{2+p}}\leq C (\|U\|_{\calPM^{(n-1)-p}} + \|U\|_{\calX^{2+p}})(\|W\|_{\calPM^{(n-1)-p}} + \|W\|_{\calX^{2+p}}).
\end{align}

The Fourier coefficients of $(\calL V_0)(t,\cdot)$ are given by:
\begin{equation}\label{FouriercalL}
  \mathcal{F}(\calL V_0)(t,k) = e^{g(t)|k| -t \sigma(\Delta^2+\Delta)(k)}\hat{V_0}(k).
\end{equation}
The Fourier coefficients of the nonlinear term are
\begin{align}\label{FouriercalB}
 \nonumber \mathcal{F}&(\calB (U,W))(t,k) = -\int_0^t e^{g(t)|k| -(t-s) \sigma(\Delta^2+\Delta)(k)}\Big[ \\
  & \sum_{j \in \mathbb{Z}_\ast^n, j \neq k}  ((k - j) \cdot j) \, e^{-g(s)|k-j|}\hat{U}(s,k-j)  e^{-g(s) |j|} \hat{W}(s,j)\Big] \, ds.
\end{align}

In what follows we will consider two kinds of temporal weights:
\begin{align}
\label{fourthroot} & g(t) = a \sqrt[4]{t}, \text{ for some constant } a>0;\\
\label{lintime} & g(t) = bt, \text{ for some constant } b>0.
\end{align}

In order to estimate the linear term $\calL$ we will make use of the following technical lemma.

\begin{lemma} \label{goftLinEst}
  Let $k \in \Omega_I$. Then, if $M_2$ is as in \eqref{M2}, it holds that:
    \begin{enumerate}
        \item if $g(t) = bt$, with $b < \frac{M_2}{2}$, then
         \[g(t) |k| - t\sigma(\Delta^2+\Delta)(k) \leq - \frac{M_2 t}{2}|k|^4;\]
        \item if $g(t) = a\sqrt[4]{t}$ then there exists $C=C(a)>0$ such that
        \[g(t) |k| - t\sigma(\Delta^2+\Delta)(k) \leq C - \frac{M_2 t}{2}|k|^4,\]
        for all $t \geq 0$.
      \end{enumerate}
\end{lemma}

\begin{proof}
  Recall the definition of $M_2$, from \eqref{M2}, such that \[\sigma(\Delta^2+\Delta)(k) \geq M_2 |k|^4.\]
Let us first consider the case $g(t) = bt$ with $b < \frac{M_2}{2}$. Then, clearly, if $k$ is such that
$|k| \geq 1$, it follows that $bt|k| - \frac{M_2}{2} t |k|^4 \leq 0$. It then follows easily that,
if $|k| \geq 1$,
\[ g(t) |k| - t\sigma(\Delta^2+\Delta)(k) \leq g(t) |k| - t M_2 |k|^4 \leq - \frac{M_2 t}{2}|k|^4, \]
as desired. This establishes item (1).

  Next consider the case $g(t) = a\sqrt[4]{t}$. Then of course we have
\begin{equation}\label{proofitem2}
  g(t) |k| - t\sigma(\Delta^2+\Delta)(k) \leq a\sqrt[4]{t}|k| - M_2 t |k|^4.
\end{equation}
Consider the function $f=f(z)=az - \frac{M_2}{2}z^4$. This function is globally bounded from above. Let
\[C=C(a)=\max\{\sup f(z), 1\}.\]
Noting that $a\sqrt[4]{t}|k| - \frac{M_2}{2} t |k|^4 = f(\sqrt[4]{t}|k|)$ we obtain item (2).

\end{proof}

In view of Lemma \ref{goftLinEst} all the estimates for $\calL V_0 \equiv \calL\psi_0$, \eqref{LY-1}, \eqref{LX3}, \eqref{LPM} and \eqref{LX2andp}, can be reduced to the corresponding estimates already obtained for $S\psi_0$, namely \eqref{Spsi0estY-1} and those obtained in Lemmas \ref{firstSemigroupLemma} and \ref{secondSemigroupLemma}.

Next we prove another technical lemma, which will be used for the nonlinear term $\calB (U,W)$.

\begin{lemma} \label{goftNonLinEst}
  Let $k \in \Omega_I$. Then, if $M_2$ is as in \eqref{M2}, it holds that:
    \begin{enumerate}
         \item if $g(t) = bt$, with $b < \frac{M_2}{2}$, then it follows that
         \[(g(t) - g(s))|k| - (t-s)\sigma(\Delta^2+\Delta)(k) \leq - \frac{M_2 (t-s)}{2}|k|^4,\]
         for all $t$, $s$ such that $0\leq s \leq t$ and all $k \in \Omega_I$;
        \item if $g(t) = a\sqrt[4]{t}$ then there exists $C=C(a)>0$ such that
        \[(g(t) - g(s)) |k| - (t-s)\sigma(\Delta^2+\Delta)(k) \leq C - \frac{M_2 (t-s)}{2}|k|^4,\]
        for all $t$, $s$ such that $0\leq s \leq t$ and all $k \in \Omega_I$.

      \end{enumerate}
\end{lemma}

\begin{proof}
  Let us begin with item (1), the case $g(t)=bt$, $b < M_2/2$. In view of Lemma \ref{goftLinEst} item (1) this is trivial since
  \[(g(t) - g(s))|k| - (t-s)\sigma(\Delta^2+\Delta)(k) = (t-s)b|k| - (t-s)\sigma(\Delta^2+\Delta)(k),\]
  and $t-s \geq 0$.

  Next consider $g(t) = a\sqrt[4]{t}$. We first note that
  \[(g(t) - g(s)) |k| - (t-s)\sigma(\Delta^2+\Delta)(k) \leq (g(t) - g(s)) |k| - (t-s)M_2 |k|^4,\]
  using, again, \eqref{M2} and the fact that $t-s\geq 0$. Next we observe that
  \[(g(t) - g(s)) |k| - (t-s)M_2 |k|^4 = f(a\sqrt[4]{t})-f(a\sqrt[4]{s}) - \frac{M_2 (t-s)}{2}|k|^4,\]
  where $f$ was introduced in the proof of Lemma \ref{goftLinEst}. There are two possibilities: either $f(a\sqrt[4]{s})\geq 0$, in which case we may ignore this term and use the boundedness from above of $f$ to obtain (2), or $f(a\sqrt[4]{s})< 0$. Let us assume the latter and note that $f$ has only two real roots, namely $0$ and $\sqrt[3]{\frac{2a}{M_2}}>0$. In addition, $f$ restricted to the positive real axis is only negative on the interval $\left(\sqrt[3]{\frac{2a}{M_2}}, + \infty\right)$, on which it is also decreasing. Therefore, since $s \leq t$, we have $f(a\sqrt[4]{t}) \leq f(a\sqrt[4]{s})$, from which (2) follows immediately.

\end{proof}

Now we re-write \eqref{FouriercalB} in a more convenient form and estimate:
\begin{align}\label{FouriercalBconvenient}
 \nonumber |\mathcal{F}&(\calB (U,W))(t,k)| =
 \left|-\int_0^t e^{(g(t)-g(s))|k| -(t-s) \sigma(\Delta^2+\Delta)(k)}\Big[ \right.\\
  & \left. \sum_{j \in \mathbb{Z}_\ast^n, j \neq k}  ((k - j)\cdot j) \, e^{g(s)(|k|-|k-j|-|j|)}\hat{U}(s,k-j)  \hat{W}(s,j)\Big] \, ds \right| \\
 \label{goodEstNonLin} & \leq \int_0^t e^{(g(t)-g(s))|k| -(t-s) \sigma(\Delta^2+\Delta)(k)}\Big[
  \sum_{\substack{ j \in \mathbb{Z}_\ast^n \\ j \neq k}}|k-j||j||\hat{U}(s,k-j) | |\hat{W}(s,j)|\Big] \, ds,
\end{align}
where we used the triangle inequality to estimate $|k|-|k-j|-|j|\leq 0$.

In view of Lemma \ref{goftNonLinEst} it is easy to see that the  estimates on the term \eqref{goodEstNonLin}, namely \eqref{BY-1}, \eqref{BX3}, \eqref{BPM} and \eqref{BX2andp}, can be reduced to the corresponding ones for $B(F,G)$, \eqref{BFGY-1}, \eqref{BFGestY-1X3}, \eqref{BFGPM-p}, \eqref{BFGX2andp1}, \eqref{BFGPM1lessp} and \eqref{BFGX2andp2}, established in the previous sections.

We now comment on how these results imply analyticity of solutions.  By the periodic analogue of Theorem IX.13
of \cite{reedSimon}, a function is analytic with radius of analyticity at least $\rho$ if its Fourier series decays
like $e^{-\tilde{\rho}|k|}$ for all $\tilde{\rho}<\rho.$  With a solution $V\in\mathcal{Y}^{-1},$ then at each time $t$,
we have,  for any $\varepsilon>0,$ the existence of $c>0$ such that
\begin{equation}\nonumber
c\sum_{k\in\mathbb{Z}^{n}_{*}}e^{(g(t)-\varepsilon)|k|}|\hat{\psi}(t,k)|
\leq \sum_{k\in\mathbb{Z}^{n}_{*}} e^{(g(t)-\varepsilon)|k|}\frac{e^{\varepsilon|k|}}{|k|}|\hat{\psi}(t,k)|
\leq \|V\|_{\mathcal{Y}^{-1}}.
\end{equation}
We see from this that $\hat{\psi}(t,\cdot)$ decays like $e^{-\tilde{\rho}|k|}$ for any $\tilde{\rho}<g(t),$ and thus
$\psi$ is analytic with radius of analyticity at least $g(t).$  Similarly, for solutions with $V\in\mathcal{PM}^{(n-1)-p},$
the solution of Kuramoto-Sivashinsky, $\psi,$ is again analytic with radius of analyticity at least $g(t).$

By virtue of these considerations we have established the following results.

\begin{theorem}\label{firstAnalyticityTheorem}
 Let $T>0$ be given.  (If the conditions of Case A hold, then $T$ may be
taken to be $T=\infty.$)
Let $n\geq1.$  There exists $\varepsilon>0$ such that for any $\phi_{0}$ with $\mathbb{P}\phi_{0}\in Y^{-1},$ if
$\|\mathbb{P}\phi_{0}\|_{Y^{-1}}<\varepsilon,$ then there exists  $\phi$ with 
$\mathbb{P}\phi\in\mathcal{Y}^{-1}\cap \mathcal{X}^{3}$ such that $\phi$ is an analytic mild solution
to the initial value problem \eqref{KSEquationOriginal}, \eqref{data} with radius of analyticity at least $R(t) = \max\{a\sqrt[4]{t},bt\}$, with $b<M_2$ and $a > 0$.
\end{theorem}

\begin{theorem}\label{secondAnalyticityTheorem}
Let $n \in \{1, \, 2\}$. Let $p\in(0,1/2)$  and $T>0$ be given.
(If the conditions of Case A hold, then $T$ may be
taken to be $T=\infty.$)
There exists $\varepsilon>0$ such that for any $\phi_{0}$ with $\mathbb{P}\phi_{0}\in PM^{(n-1)-p},$ if
$\|\mathbb{P}\phi_{0}\|_{PM^{(n-1)-p}}<\varepsilon,$ then there exists $\phi$ with
$\mathbb{P}\phi\in\mathcal{PM}^{(n-1)-p}\cap \mathcal{X}^{2+p}$ such that $\phi$ is an analytic mild solution
to the initial value problem \eqref{KSEquationOriginal}, \eqref{data} with radius of analyticity at least $R(t) = \max\{a\sqrt[4]{t},bt\}$, with $b<M_2$ and $a > 0$.
\end{theorem}

\section{Concluding remarks}\label{conclusionSection}

We close now with a few remarks.

First, we comment on our bound for the radius of analyticity.  We have shown that our solutions have radius of
analyticity which grows at least like $t^{1/4},$ and also at least like $t.$  Of course, the rate $t^{1/4}$ is faster
for times near zero, and the rate $t$ is faster for large times.  A fractional-power rate has previously been
observed for the Navier-Stokes equations (where the rate is $t^{1/2}$) and
for the Kuramoto-Sivashinsky equation (where the rate is $t^{1/4}$) for solutions on $\mathbb{R}^{n}$
\cite{baePAMS}, \cite{foiasTemam}, \cite{grujicKukavica}.
For spatially periodic problems, rates like $t$ have been observed previously for the Navier-Stokes equations
\cite{foiasTemam} or for more general parabolic equations \cite{ferrariTiti}.  We have previously observed
for the Navier-Stokes equations that in the periodic case, one gets the improvement that both of these rates hold
\cite{ALN4}.  The present work shows this improvement in the periodic case of the Kuramoto-Sivashinsky equation.
The radius of analyticity of solutions is relevant for the convergence rate of numerical simulations \cite{doelmanTiti}.

Let $n \in \{1,2\}$. Then we note that the two function spaces for the initial data
we have considered in this work, $Y^{-1}$ and $PM^{(n-1)-p}$, with $p<1/2$,
are not comparable.  Consider, for instance, $f$  such that
\[|\hat{f}(k)|=\frac{1}{|k|^{n-1}}, \text{ for all } k.\]
Then $f\in PM^{(n-1)-p},$ but $f\notin Y^{-1}.$
On the other hand, let $f$  be such that
\[
|\hat{f}(j)| =
\left\{
\begin{array}{ll}
|j|^{3/4} & \text{ if } |j|= 2^{4\ell}, \; \ell = 1, 2, \ldots \\
0 & \text{ otherwise. }
\end{array}
\right.
\]
Then $f\in Y^{-1}$ but $f\notin PM^{-1/2},$
and thus $f\notin PM^{(n-1)-p}.$

In the introduction, we mentioned the Navier-Stokes results of Koch and Tataru \cite{kochTataru}, Cannone and Karch
\cite{cannoneKarch}, and Lei and Lin \cite{leiLin} as works proving existence of solutions for the Navier-Stokes equations
with low-regularity data.  It should be noted that the function spaces considered in the aforementioned works, $BMO^{-1},$ $PM^{2},$ and $X^{-1},$ respectively,
are all critical
spaces for the Navier-Stokes equations.

If we discard the unstable Laplacian and consider \eqref{KSEquationOriginal} in full space $\real^n$ then it is easy to see that this modified Kuramoto-Sivashinsky equation is invariant under the scaling 
\[\lambda \mapsto \lambda^2 \psi(\lambda x, \lambda^4 t).\]
Thus, among the hierarchy of spaces considered in this work, the spaces $Y^{-2}$ and $PM^{n-2}$ are {\it critical spaces}, i.e. whose norms are invariant under this scaling.
In the present work we lower the regularity requirements for existence theory for the Kuramoto-Sivashinsky
equation as compared to the prior literature, 
and have proved existence of solutions in spaces of negative index, namely $Y^{-1}$ and $PM^{(n-1)-p}$, for any $0<p<1/2$, but these spaces are not critical.  Therefore
there remains work to be done to continue lowering the regularity threshold for the initial data.

We also mentioned in the introduction that solutions of the Kuramoto-Sivashinsky equation have been proved to be global
in one spatial dimension, when starting from $H^{1}$ data.  In the present work we have shown existence of solutions with
rough data, but only until a short time (unless the spatial domain $[0,L]$ satisfies $L<2\pi$).  But we have shown that the
solutions are analytic at positive times, and thus the solutions instantaneously become $H^{1}$ solutions which could then
be continued for all time.  So, our one-dimensional solutions are in fact global.  However, the present method would not
extend on its own to demonstrate this.  The radius of analyticity that we prove grows in time, like both $t^{1/4}$ and
$t.$  This growth of the radius for all time is possible in the small-domain case (again, $L<2\pi$), but in the presence of
linearly growing modes ($L>2\pi$), one would not expect this.  Instead, the solution in some cases tends toward
coherent structures such as traveling waves or time-periodic waves, and these attracting solutions tend to have
finite radius of analyticity.  The long-time behavior of the radius of analyticity for the initial value problem, then, is to
tend towards this value of the radius of analyticity rather than to tend towards infinity.  This can be seen from
computational work such as \cite{papageorgiou1}, \cite{papageorgiou2}.  Understanding in more detail the time
evolution of the radius of analyticity of solutions of the Kuramoto-Sivashinsky problem will be a subject of future work.

\section*{Acknowledgments}
The second and third authors gratefully acknowledge the hospitality of the Department of Mathematics at Drexel University, where part of this research was done.

DMA gratefully acknowledges support from the National Science Foundation through grants DMS-1907684 and
DMS-2307638. MCLF was partially supported by CNPq, through grant \# 304990/2022-1, and FAPERJ, through  grant \# E-26/201.209/2021.
HJNL acknowledges the support of CNPq, through  grant \# 305309/2022-6, and of FAPERJ, through  grant \# E-26/201.027/2022.

\section*{Data availability statement}
Data sharing not applicable to this article as no datasets were generated or analysed during the current study.

\bibliography{KSexistAn.bib}{}
\bibliographystyle{plain}

\end{document}